\documentclass[a4paper]{article}
\usepackage[dvipdf]{graphicx}
\usepackage{amsthm}
\newtheorem{lem}{Lemma}
\newtheorem{thm}[lem]{Theorem}
\newtheorem{df}[lem]{Definition}
\newtheorem{cor}[lem]{Corollary}

\newtheorem{prop}[lem]{Proposition}
\newtheorem{ex}[lem]{Example}
\newtheorem{con}[lem]{Conjecture}

\title{Sports scheduling for not all pairs of teams}
\author{Kenji KASHIWABARA, \\
Department of General Systems Studies, University of Tokyo, \\
3-8-1 Komaba, Meguroku, Tokyo, Japan}
\date{}

\begin{document}
\maketitle

\begin{abstract}
We consider the following sports scheduling problem. Consider $2n$ teams in a sport league. Each pair of teams must play exactly one match in $2n-1$ days. That is, $n$ games are held simultaneously in a day. We want to make a schedule which has $n(2n-1)$ games for $2n-1$ days. 

When we make a schedule, the schedule must satisfy a constraint according to the HAP table, which designates a home game or an away game for each team and each date.  Two teams cannot play against each other unless one team is assigned to a home game and the other team is assigned to an away game. Recently, D. Briskorn proposed a necessary condition for a HAP table to have a proper schedule. And he proposed a conjecture that such a condition is also sufficient. That is, if a solution to the linear inequalities exists, they must have an integral solution. In this paper, we rewrite his conjecture by using perfect matchings. We  consider a monoid in the affine space generated by perfect matchings. In terms of the Hilbert basis of such a monoid, the problem is naturally generalized to a scheduling problem for not all pairs of teams described by a regular graph. In this paper, we show a regular graph such that the corresponding linear inequalities have a solution but do not have any integral solution. Moreover we discuss for which regular graphs the statement generalizing the conjecture holds.
\end{abstract}

\section{Introduction}

First, consider the situation that there are $2n$ teams in a sport league and we have to make a schedule for any pair of teams to play exactly one match in $2n-1$ days. 
$n$ games are held simultaneously everyday. We have to make a schedule which have $n(n-1)$ games in $2n-1$ days. Each team plays against every other team. Such a tournament method is called a Round Robin Tournament in the literature. 

The schedule that we consider must obey the following constraint. The schedule must be compatible with the given table which defines the stadium availability for each day and each team. Such a table is called a home and away pattern(HAP) table. A HAP table has an entry of a home game(H) or an away game(A) for each day and each team. A home game is a game at the own stadium, and an away game is a game at the stadium  of the opponent team.  For a team and a date, the HAP table gives one of H and A. We assume that two teams cannot play against each other for a day unless one team is assigned to a home game and the other team is assigned to an away game on that day. On a day, a team which is assigned to a home game is called a home team, and a team which is assigned to an away game is called an away team.

The existence of a schedule that satisfies the property above depends upon a HAP table. To begin with, we consider the following problem. What HAP table has a schedule compatible with the HAP table for any pair of teams to play against each other?
A schedule compatible with the HAP table means a schedule satisfying the constraint of the HAP table.

For example, we consider the HAP table with four teams in Table \ref{tab:hap}.

\begin{table}
\begin{center}
\begin{tabular}{|c|c|c|c|}
\hline
   & 1st day & 2nd day & 3rd day \\
\hline
team 1 & H & H & A \\
\hline
team 2 & H & A & H \\
\hline
team 3 & A & A & A \\
\hline
team 4 & A & H & H \\
\hline
\end{tabular}
\caption{HAP table for four teams}\label{tab:hap}
\end{center}
\end{table}

\vspace{4mm}

A pair $\{1,2\}$ of teams cannot play against each other on the first day because both teams are home teams on the first day. But $\{1,3\}$ can play against each other on the first day because team 1 is assigned to a home game and team 3 is assigned to an away game.

Consider the schedule such that two matches $\{1,3\},\{2,4\}$ are held on the first day, two matches $\{1,2\},\{3,4\}$ are held on the second day, and two matches $\{1,4\},\{2,3\}$ are held on the third day. This schedule satisfies the constraint that any pair of two teams which play against each other is a pair of a home team and an away team.

The problem that we consider here is a kind of sports scheduling problems. For details, see D. de
Werra\cite{Werra80,Werra82,Werra85,Werra88} and R. Miyashiro, H. Iwasaki, and T. Matsui\cite{miya}.

R. Miyashiro, H. Iwasaki, and T. Matsui\cite{miya} proposed a necessary condition for a HAP table to have a schedule compatible with the HAP table, but it is not sufficient. Recently, D. Briskorn\cite{dirk0,dirk} proposed a new necessary condition for a HAP table to have a compatible schedule. This necessary condition is described by linear inequalities.  He conjectured that such a condition is also sufficient. His conjecture is that the linear inequalities must have an integral solution whenever they have a non-integral solution. In this paper, we rewrite his conjecture in terms of perfect matchings. 
Rewriting the conjecture in terms of perfect matchings gives us a method to attack the conjecture by computer calculation. 
We confirm the conjecture for the complete graph on 6 vertices in Theorem \ref{thm:nocount6} by computing the Hilbert basis of an affine monoid generated by perfect matchings. 

By using the Hilbert basis, the problem is naturally generalized to a schedule for not all pairs of teams. While the complete graph corresponds to the scheduling problem for all pairs of teams, a regular graph corresponds to the scheduling problem for not all pairs of teams. For example, consider the league consisting of the teams in the west league and the teams in the east league, and the tournament such that each team in the west league and each team in the east league should play against each other exactly once. Such a tournament is represented by the complete bipartite graph.

We show that there exists a regular graph to which the corresponding problem has a non-integral solution but does not have any integral solution. That is, the statement generalizing the conjecture does not hold for some graphs. We discuss which regular graph satisfies the statement generalizing the conjecture. We show that any antiprism on  even vertices does not satisfy the statement generalizing the conjecture in Theorem \ref{thm:antiprism}. We also give a cubic bipartite graph which does not satisfy the statement generalizing the conjecture in Example \ref{ex:cubicbib}.

\section{Basic formulation}

\subsection{Formulation of scheduling problems}

Let $V$ be the finite set of vertices of a graph that we consider. $V$ is interpreted as the set of teams in the sports league. $|V|$ is always assumed to be even.

A partition of a set that consists of two sets of the same size is called an equal partition in this paper. The set of all the equal partitions on $V$ is denoted by $C=C(V)$. 
It is not important which indicates home games and which indicates away games in the two sets because that is irrelevant to whether there exists a solution or not. An equal partition can be identified with a complete bipartite graph which has two partite sets of the same size. We denote the complete bipartite graph corresponding to $c\in C$ by $B_c$. 

Let $K:= {V \choose 2}$, which is identified with the edges of the complete graph on $V$. We consider the linear space ${\mathbf R}^{K\cup C}$ of dimension $|K\cup C|$. For $v \in {\mathbf R}^{K\cup C}$, the components of $v$ in $K$ are called the edge components and the components of $v$ in $C$ are called the HA components. For $v\in {\mathbf R}^{K\cup C}$, $v|_K$ is defined to be the vector restricted to the edge components.

Denote $E(v)=\{\{a,b\}\in K|v(\{a,b\})=1\}$.

For graph $G=(V,E)$ and $c\in C$, we introduce a vector $\chi_{E,c}\in {\mathbf N}^{K\cup C}$ defined as follows, where $\mathbf N$ is the set of nonnegative integers.

For $e\in K$, let

$$\chi_{E,c}(e)=\left\{
\begin{array}{c}
1 ... e\in E(G)\\
0 ... e\notin E(G)
\end{array}
\right..$$

For $f\in C$, let

$$\chi_{E,c}(f)=\left\{
\begin{array}{c}
1 ... f=c\\
0 ... f\neq c
\end{array}
\right..$$

$\chi_E$ means the vector obtained by the restriction of $\chi_{E,c}$ to the edge components. That it, $\chi_E$ is the characteristic function of edges $E$.

Let $$PM(V)=\{\chi_{q,c}|q\mbox{ is a perfect matching of }B_c \mbox{ for some }c\in C\}.$$

Recall that $B_c$ is the complete bipartite graph whose partite sets are $c\in C$. $PM(V)$ plays an important role in this paper.
Note that $PM(V)$ does not depend upon a certain graph but only upon $V$.
We also denote $\chi_{q,c} \in PM(V)$ by $(q,c)\in PM(V)$ for simplicity.

\begin{ex}
The size of $PM(\{1,2,3,4\})$ is 6. Figure \ref{fig:pm4} illustrates all six vectors in $PM(\{1,2,3,4\})$. For example, the first figure illustrates  vector $v\in PM(\{1,2,3,4\})$ such that $v(\{\{1,2\},\{3,4\}\})=1$ and $v(\{1,3\})=v(\{2,4\})=1$ and the value of $v$ on any other set takes 0.
\begin{figure}[ht]
\begin{center}
\includegraphics[width=12cm]{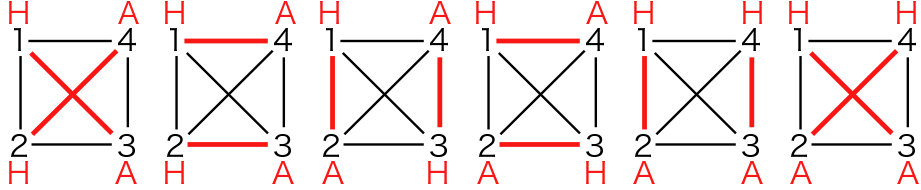}
\caption{Elements of $PM(\{1,2,3,4\})$}\label{fig:pm4}
\end{center}
\end{figure}

\end{ex}

We confine ${\mathbf N}^{K\cup C}$ to more meaningful vectors as follows.

Let 
$$\mbox{problem}(V)=\{v\in {\mathbf N}^{K\cup C}|\ v|_K \leq \chi_K,E(v)\mbox{ is a regular graph,}\sum_{c\in C}v(c)\times \frac{|V|}{2}=\sum_{e\in K}v(e)\},$$
where $\chi_K$ is the function whose value is 1 on any pair of vertices. 
Recall that an $r$-regular graph is a graph where each vertex has $r$ neighbors.

The last condition of the definition of problem$(V)$ is to balance between the HA components and edges components. This condition is required by the assumption that $\frac{|V|}{2}$ games are held in a day. By definition, $PM(V)\subset \mbox{problem}(V)$ holds. Moreover, note that the linear combinations $v$ of $PM(V)$ with $v|_K \leq \chi_K$ are also included in $\mbox{problem}(V)$.

For $v\in \mbox{problem}(V)$, $v|_K$ gives the pairs of teams to play against each other and $v|_C$ gives the HAP table because $v(c)$ indicates the number of appearances of equal partition $c\in C$ in the HAP table. So giving $v\in \mbox{problem}(V)$ means giving a scheduling problem to consider.

\begin{ex}
We consider a scheduling problem for four teams and three days with the HAP table in Table \ref{tab:hap}. The vector $v\in \mbox{problem}(\{1,2,3,4\})$ corresponding to this HAP table is given by
$$v(\{1,2\})=v(\{1,3\})=v(\{1,4\})=v(\{2,3\})=v(\{2,4\})=v(\{3,4\})=1,$$
$$v(\{\{1,2\},\{3,4\}\})=v(\{\{1,4\},\{2,3\}\})=v(\{\{1,3\},\{2,4\}\})=1.$$
\end{ex}

\subsection{Two monoids}

For ${\mathcal M}\subset {\mathbf N}^{K\cup C}$,  ${\mathbf N}({\mathcal M})$ denotes the set of all nonnegative integral combinations of ${\mathcal M}$. That is, 

$${\mathbf N}({\mathcal M})=\{\sum k_ip_i|k_i \in {\mathbf N},p_i\in {\mathcal M}\},$$

where ${\mathbf N}$ is the set of non-negative integers.

When $v\in {\mathbf N}(PM(V))\cap\mbox{problem}(V)$ is given, we can associate $d\in \{1,\ldots,r\}$ with $p_d\in PM(V)$ so that the HA component of $p_d$ is $\{H(d),A(d)\}$ where $v=\sum_{p_d\in PM(V)} k_dp_d$. 
For $v\in\mbox{problem}(V)$, we can interpret $v\in {\mathbf N}(PM(V))$ as indicating that there exists a schedule compatible with the HAP table for the pairs of teams given by $E(v)$. For $v\in {\mathbf N}(PM(V))\cap\mbox{problem}(V)$, we say that $r$-regular graph $E(v)$ is $r$-edge-colorable compatible with the HAP table because, for every $d\in \{1,\ldots,r\}$, each vertex is incident with exactly one edge in perfect matching $q_d$ with $p_d=(q_d,\{H(d),A(d)\})$ using the correspondence induced by $v$.

Let $\overline{\mathbf N}({\mathcal M})$ be defined as 
$$\overline{\mathbf N}({\mathcal M})=\{v\in {\mathbf N}^{K\cup C}|k v\in {\mathbf N}({\mathcal M}) \mbox{ for some }k \geq 1\}.$$

We always take $PM(V)$ as ${\mathcal M}$ in this paper. By definition, we have ${\mathbf N}(PM(V))\subset \overline{\mathbf N}(PM(V))$. $\overline{\mathbf N}(PM(V))$ is the set of integral points in the convex hull of ${\mathbf N}(PM(V))$.

Note that ${\mathbf N}(PM(V))$ and $\overline{\mathbf N}(PM(V))$ do not depend upon a certain graph but only upon $V$.

Moreover, note that both of ${\mathbf N}(PM(V))$ and $\overline{\mathbf N}(PM(V))$ are closed under addition with $0$. That is, both are monoids in ${\mathbf N}^{K\cup C}$.

\begin{prop}
If $v \in \overline{\mathbf N}(PM(V))$ and $v|_K \leq \chi_K$, $v\in \mbox{problem}(V)$ holds.
\end{prop}

\begin{proof}
$v\in PM(V)$ satisfies $\sum_{c\in C}v(c)=\sum_{e\in K}v(e)\times 2/|V|$.
Moreover, $E(v)$ is a 1-regular graph for $v\in PM(V)$. $\overline{\mathbf N}(PM(V))$ is generated by nonnegative combinations of $PM(V)$. So, for $v \in \overline{\mathbf N}(PM(V))$, we have $\sum_{c\in c}v(c)=\sum_{e\in K}v(e)\times 2/|V|$. Since $v|_K \leq \chi_K$, $E(v)$ is a regular graph.
\end{proof}

\subsection{B-factorizability}

%Note that the converse statement does not hold. That is, for some regular graph $G$, there exists no $v\in \overline{\mathbf N}(PM(V))$ such that $v|_K=\chi_E$ since it is known that some regular graph has no perfect matching.

%In our setting of the scheduling problem, a regular graph gives us the pairs of teams to play against each other. 

\begin{df}
A regular graph $G=(V,E)$ is called B-factorizable if any $v\in \overline{\mathbf N}(PM(V))$ with $v|_K=\chi_E$ satisfies $v\in {\mathbf N}(PM(V))$.
\end{df}

In other words, an $r$-regular graph is a B-factorizable graph when, for any $v \in \mbox{problem}(V)$ such that $E(v)$ coincides with the graph, $kv\in {\mathbf N}(PM(V))$ for some $k\geq 1$ implies that there exists an $r$-edge-coloring compatible with the HAP table of $v$.

Note that a B-factorizable graph $(V,E)$ satisfies
$$\{v\in \overline{\mathbf N}(PM(V)):\ v|_K=\chi_E\}=\{v\in {\mathbf N}(PM(V)):\ v|_K=\chi_E\}.$$

It is easy to see that a 2-regular graph is B-factorizable although a 2-regular graph with an odd cycle has no $v\in \overline{\mathbf N}(PM(V))$ with $v|_K=\chi_E$.

The next conjecture is equivalent to Conjecture \ref{con:briskorn}, proposed by D. Briskorn. The equivalence will be proved in Corollary \ref{cor:main}.

\begin{con}\label{con:dirk}
The complete graph $K_{|V|}$ is B-factorizable.
\end{con}

When $|V|=4$, the conjecture holds because of $\overline{\mathbf N}(PM(V))={\mathbf N}(PM(V))$.

\section{Briskorn's conjecture}

\subsection{Briskorn's conjecture}

We state the conjecture proposed by D. Briskorn\cite{dirk0,dirk}. D. Briskorn considered the scheduling problem for all pairs of teams, but we generalize his framework to a scheduling problem for the pairs of teams given by a regular graph.

For a set $V$ of teams, consider $r$-regular graph $(V,E)$ on $V$.   Consider a schedule for $r$ days. A HAP table for $r$ days is given in the form of  the following pair of functions $H$ and $A$.
$$(H,A):\{1,\ldots,r\}\to C\quad (d\mapsto \{H(d),A(d)\}).$$
Recall that $C$ is the set of the equal partitions on $V$.

We consider $x_{\{a,b\},d}$ as a variable for pair $\{a,b\}\in K$ of teams and date $d\in \{1,\ldots,r\}$. 
A variable $x_{\{a,b\},d}$ takes 1 when team $a$ and team $b$ play against  each other on $d$-th day, and takes 0 when they do not.

We suppose that $x$ must satisfy the following conditions.

\begin{enumerate}
\item For any $\{a,b\}\in E$,
$$\sum_{d\in \{1,\ldots,r\}}x_{\{a,b\},d}=1.$$
For any $\{a,b\} \notin E$, $x_{\{a,b\},d}=0$.
\item For any $d\in \{1,\ldots,r\}$ and any $a\in V$,
$$\sum_{b\in V:b\neq a} x_{\{a,b\},d}=1.$$
\item For any $d\in \{1,\ldots,r\}$ and any $\{a,b\}\in E$,
$$0 \leq x_{\{a,b\},d} \leq 1.$$
\item For any $d\in \{1,\ldots,r\}$ and any $\{a,b\}\in E$,
$(a,b)\in (H(d)\times H(d))\cup (A(d)\times A(d))$ implies $x_{\{a,b\},d}=0$.
\end{enumerate}

Condition 1 is interpreted as that every pair of teams given by $E$ should play against each other exactly once.
Condition 2 is interpreted as that, on any date, any team plays against another team exactly once.
Condition 3 is to relax the integer programming problem to the linear programming problem.
Condition 4 is interpreted as that any pair of teams assigned to both H or both A cannot play against each other.

The conditions above are all described by linear inequalities, so the set of solutions which satisfy the linear inequalities forms a polytope. 
This polytope is determined by a graph and a HAP table, so we write this polytope as $P(G,HA)$.

When $P(G,HA)$ has an integral point, its components consist of 0 and 1 by Condition 3. So such a solution gives a desired schedule. In that schedule, team $a$ and team $b$ play against each other on $d$-th day when $x_{\{a,b\},d}=1$.

The necessary condition, proposed by D. Briskorn, for a HAP table to have a proper schedule is that $P(K_{|V|},HA)$ is non-empty. He conjectured that this necessary condition is also sufficient\cite{dirk0,dirk}.

\begin{con}\label{con:briskorn}
Consider the complete graph $K_{|V|}$. For any HAP table, if $P(K_{|V|},HA)$ is non-empty, it must contain an integral point.
\end{con}

\subsection{Equivalence between two conjectures}

We want to rewrite this conjecture in terms of perfect matchings. We prove the equivalence between Conjecture \ref{con:dirk} and Conjecture \ref{con:briskorn} in Corollary \ref{cor:main}.

\begin{lem}\label{lem:nonint}
For regular graph $G=(V,E)$ and any HAP table, $P(G,HA)$ is non-empty if and only if there exists $v\in \overline{\mathbf N}(PM(V))$
 such that $v|_K=\chi_E$ and $v(c)=|\{d\in \{1,\ldots,r\}:c=\{H(d),A(d)\}\}|$ for any $c\in C$.
\end{lem}

\begin{proof}
Let $x\in P(G,HA)$. So $x$ satisfies Conditions 1, 2, 3, and 4. By Conditions 2, 3, 4 and well-known arguments about defining inequalities of matching polytopes for bipartite graphs\cite{LP}, for any $d\in \{1,\ldots,r\}$, $x_{\cdot,d}$ is contained in the convex-hull of the perfect matchings of the complete bipartite graph whose partite sets are $\{H(d),A(d)\}$. Therefore, it can be represented by a convex combination of perfect matchings of the bipartite graph. We can write $x_{\{a,b\},d}=
(\sum_{q:(q,\{H(d),A(d)\})\in PM(V)} s_{q,d} {\bf q})(\{a,b\})$ where $s_{q,d}$ is a nonnegative real number. ${\bf q}$ means the characteristic function $\chi_q$ of a perfect matching $q$ for simplicity. Since $(q,\{H(d),A(d)\})\in PM(V)$, $q$ is a perfect matching compatible with the HAP table on $d$-th day. Then 
$$x_{\{a,b\},d}=\sum_{(q,\{H(d),A(d)\})\in PM(V)} s_{q,d} {\bf q}(\{a,b\})=\sum_{{(q,\{H(d),A(d)\})\in PM(V)}\atop {q(\{a,b\})=1}}s_{q,d}.$$

We want to take $v$ so that $v|_K=\chi_E$ and $v(c)=|\{d\in \{1,\ldots,r\}:c=\{H(d),A(d)\}\}|$ for any $c\in C$. 
We define $k_p = k_{(q,c)}:= \sum_{d:\{H(d),A(d)\}=c} s_{q,d}$ for $p=(q,c)\in PM(V)$. Let $v=\sum_{p\in PM(V)} k_p p$. 
We make sure that $v$ satisfies the desired condition.

Then the value of $v(\{a,b\})$ for $\{a,b\}\in K$ is
$$(\sum_{p\in PM(V)} k_p p)(\{a,b\})=\sum_{(q,c)\in PM(V)\atop {q(\{a,b\})=1}} k_{(q,c)} = \sum_{d\in \{1,\ldots,r\} \atop{q(\{a,b\})=1}} s_{q,d} = \sum_{d\in \{1,\ldots,r\}} x_{\{a,b\},d}.$$
When $\{a,b\}\in E$, this value is 1 by Condition 1. When $\{a,b\}\notin E$, this value is 0 by Condition 1. So we have $v|_K=\chi_E$.

On the other hand, for $c\in C$, $v(c)$ is 
\begin{eqnarray*}
(\sum_{p\in PM(V)}k_p p)(c) &=& \sum_{q:(q,c)\in PM(V)}k_{(q,c)}=\sum_{{q:(q,c)\in PM(V)}\atop{d:\{H(d),A(d)\}=c}} s_{q,d}=\sum_{d:\{H(d),A(d)\}=c} \sum_{b:q(\{a,b\})=1\atop{q:(q,c)\in PM(V)}}s_{q,d}{\bf q}(\{a,b\})\\&=&\sum_{b\in V-\{a\} \atop {d:\{H(d),A(d)\}=c}}x_{\{a,b\}, d}
 = |\{d\in \{1,\ldots,r\}:c=\{H(d),A(d)\}\}|,
\end{eqnarray*} 
by Condition 2, where $a\in V$ is an arbitrary fixed element.

Conversely, we assume that there exists $v\in \overline{\mathbf N}(PM(V))$ such that $v|_K=\chi_E$ and $v(c)=|\{d\in \{1,\ldots,r\}:c=\{H(d),A(d)\}\}|$. We may write $v=\sum_{p\in PM(V)} k_p p$ because of $v\in \overline{\mathbf N}(PM(V))$. We define 
\begin{eqnarray*}\label{eqn:star}
x_{\{a,b\},d}=\sum_{{p\in PM(V),p(\{a,b\})=1}\atop{p(\{H(d),A(d)\})=1}}k_p/v(\{H(d),A(d)\}).
\end{eqnarray*}
Note that the denominator cannot be 0 because of the assumption of $v(c)$. We show that $x\in P(G,HA)$ by checking Conditions 1, 2, 3, and 4.

Condition 1: For $\{a,b\}\in E$,
$$\sum_{d\in \{1,\ldots,r\}}x_{\{a,b\},d}=\sum_{p\in PM(V) \atop{p(\{a,b\})=1}}k_p=\sum_{p\in PM(V)}k_pp(\{a,b\})=v(\{a,b\})=1$$
since $v(c)=|\{d\in \{1,\ldots,r\}:c=\{H(d),A(d)\}\}|$.

Condition 2: Fix one vertex $a\in V$. For any perfect matching, vertex $a\in V$ is incident with exactly one edge in the matching. Therefore, for $d\in \{1,\ldots,r\}$, by letting $c =\{H(d),A(d)\}$,
\begin{eqnarray}
\sum_{b\in V:b\neq a} x_{\{a,b\},d}&=&\sum_{b\in V:b\neq a}\sum_{{p\in PM(V),p(\{a,b\})=1}\atop{p(c)=1}}k_p/v(c)\\
&=&\sum_{p\in PM(V) \atop {p(c)=1}}k_p/v(c)=\sum_{p\in PM(V)} k_p p(c)/v(c)=v(c)/v(c)=1.
\end{eqnarray}

Condition 3: By definition, we have $x_{\{a,b\},d}\geq 0$. Since $\sum_{d\in \{1,\ldots,r\}}x_{\{a,b\},d}=1$ for $\{a,b\}\in E$, we have $x_{\{a,b\},d}\leq 1$ for $\{a,b\}\in E$ and $d\in \{1,\ldots r\}$.

Condition 4: When $(a,b)\in H(d)\times H(d)$, there does not exist $p\in PM(V)$ such that $p(\{a,b\})=1$ and $p(\{H(d),A(d)\})=1$ because of the definition of $PM(V)$.
\end{proof}

\begin{lem}\label{lem:int}
For any $r$-regular graph $G=(V,E)$ and any HAP table, $P(G,HA)$ has an integral point if and only if there exists $v\in {\mathbf N}(PM(V))$
 such that $v|_K=\chi_E$ and $v(c)=|\{d\in \{1,\ldots,r\}:c=\{H(d),A(d)\}\}|$ for any $c\in C$.
\end{lem}

\begin{proof}
Suppose that $x\in P(G,HA)$ is integral. For $d\in \{1,\ldots,r\}$, $\{\{a,b\}\in K|x_{\{a,b\},d}=1\}$ is a perfect matching of $G$. For such a perfect matching $q$, let $s_{q,d}=1$. For other perfect matching $q$, let $s_{q,d}=0$. Note that $(q,\{H(d),A(d)\})\in PM(V)$ when $s_{q,d}=1$. Then we take
$k_{(q,c)}=\sum_{d:\{H(d),A(d)\}=c}s_{q,d}$ and $v=\sum_{p\in PM(V)}k_pp$. Then $v\in {\mathbf N}(PM(V))$ such that $v|_K=\chi_E$ and $v(c)=|\{d\in \{1,\ldots,r\}:c=\{H(d),A(d)\}\}|$ for any $c\in C$ by Conditions 1 to 4.

Conversely, assume $v\in {\mathbf N}(PM(V))$ such that $v|_K=\chi_E$ and $v = \sum_{p\in PM(V)} k_p p$ where $k_p$ is integral. Because of the assumption $v|_K=\chi_E$, $k_p$ is 0 or 1 for any $p\in PM(V)$. 
For $c\in C$, the number of $p\in PM(V)$ such that $p(c)=1$ and $k_p=1$ is $v(c)$ because of $$v(c)=(\sum_{p\in PM(V)}k_p p)(c)=\sum_{p\in PM(V)\atop p(c)=1}k_p.$$ On the other hand, for a fixed $c\in C$, the number of $d$ with $c=\{H(d),A(d)\}$ is $v(c)$ by the assumption. So we can associate $d\in \{1,\ldots,r\}$ with $p_d\in PM(V)$ injectively so that $k_{{p_d}}=1$ and $p_d(\{H(d),A(d)\})=1$. Then we define $x_{\{a,b\},d}=p_d(\{a,b\})$, whose components are integral. Since it is easy to check Conditions 1 to 4 to $x$, we have $x \in P(G,HA)$.
\end{proof}

By Lemmas \ref{lem:nonint} and \ref{lem:int}, we have the following corollary.
\begin{cor}
An $r$-regular graph is B-factorizable if and only if, for any HAP table,   $P(G,HA)$ contains an integral point whenever $P(G,HA)$ is non-empty.
\end{cor}

By applying this corollary to the complete graph, we have the following corollary.

\begin{cor}\label{cor:main}
Conjecture \ref{con:dirk} and Conjecture \ref{con:briskorn} are equivalent.
\end{cor}

\section{Hilbert basis and some classes of regular graphs}

\subsection{Hilbert basis and additional generators}

A monoid in the affine space that we consider in this paper is a set included in ${\mathbf Z}^{K\cup C}$ that is closed under addition with $0$. So $\overline{\mathbf N}(PM(V))$ and ${\mathbf N}(PM(V))$ are monoids in the affine space.

The Hilbert basis of a monoid in the affine space is introduced as follows. 
For a monoid in the affine space, a generating set is defined to be a set of integral vectors such that the nonnegative integral combinations of them are equal to the integral points which are contained in the convex-hull of the monoid. A monoid is said to be pointed when $x$ and $-x$ in the points of the monoid imply $x=0$. It is known that, for a pointed monoid,  there exists a unique minimal generating set with respect to inclusion. It is called the Hilbert basis.

A vector in the Hilbert basis of $\overline{\mathbf N}(PM(V))$ may not belong to $PM(V)$. In other words, $\overline{\mathbf N}(PM(V))$ may not be equal to ${\mathbf N}(PM(V))$ generally.

\begin{df}
We call a vector which belongs to the Hilbert basis of $\overline{\mathbf N}(PM(V))$ and does not belong to $PM(V)$ an {\it additional generator}.
\end{df}

Note that any additional generator does not belong to ${\mathbf N}(PM(V))$ because any vector in ${\mathbf N}(PM(V))$ can be divided into vectors in $PM(V)$.

The next lemma follows from the definitions of additional generators and B-factorizability. 

\begin{lem}\label{lem:addb}
When $v\in\overline{\mathbf N}(PM(V))\cap \mbox{problem}(V)$ is an additional generator, graph $(V,E(v))$ is not B-factorizable.
\end{lem}

An example showing that the converse statement does not hold will appear in Example \ref{ex:petersen1}. But we have the following lemma.

\begin{lem}\label{lem:countha}
Consider $v\in \overline{\mathbf N}(PM(V))\cap \mbox{problem}(V)$ such that  $E(v)$ is not a B-factorizable. Then there exists an additional generator $v'\in \overline{\mathbf N}(PM(V))\cap \mbox{problem}(V)$ such that $v' \leq v$.  Moreover, $E(v')$ is not B-factorizable.
\end{lem}

\begin{proof}
Consider $v\in \overline{\mathbf N}(PM(V))\cap \mbox{problem}(C)$ such that  $E(v)$ is not a B-factorizable.
Then there exists $v''\in\overline{\mathbf N}(PM(V))\cap \mbox{problem}(V)$ such that $E(v'')=E(v)$ and $v''\notin {\mathbf N}(PM(V))$. Therefore when $v''$ is expressed by a nonnegative integral combination of the Hilbert basis, $v$ cannot be expressed as a nonnegative integer combination of $PM(V)$ but some additional generator $v'$ appears in the support of the integral combinations of the Hilbert basis. Therefore $v''$ is expressed as the addition of $v'$ and some vector. So we have $v' \leq v''$. The latter statement in the lemma follows from Lemma \ref{lem:addb}.
\end{proof}

So if there exists a counterexample $v\in\overline{\mathbf N}(PM(V))$ to Conjecture \ref{con:dirk}, there exists an additional generator $v'$ of $\overline{\mathbf N}(PM(V))$ such that $E(v')$ is not a B-factorizable. This fact will be used in the proof of Theorem \ref{thm:nocount6}.

Normaliz\cite{normaliz} is a computer program which calculates the Hilbert basis from generators of a monoid. We can calculate additional generators of $\overline{\mathbf N}(PM(V))$ in terms of normaliz when $V$ is relatively small. We use normaliz to check whether a given regular graph is B-factorizable or not.

\subsection{How to represent a vector in figures}

We use a figure to show a problem $v\in \mbox{problem}(V)$ and an evidence indicating that it belongs to $\overline{\mathbf N}(PM(V))$ or ${\mathbf N}(PM(V))$.
In this subsection, we state how to represent vector $v$ in figures.

\begin{figure}[ht]
\begin{center}
\includegraphics[width=10cm]{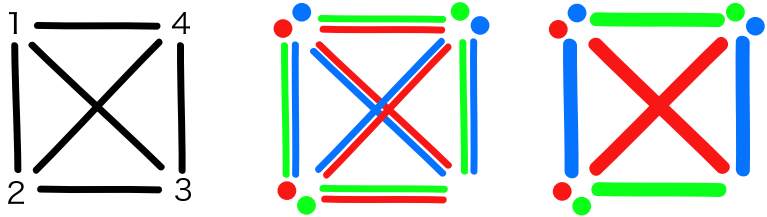}
\caption{Complete graph with 4 vertices}\label{fig:exa4}
\end{center}
\end{figure}

First we explain how to represent $v\in \mbox{problem}(V)$ by a figure.
We consider four teams and three days as an example. Then, since each pair of four teams should play against each other, the graph to be considered is the complete graph on 4 vertices. The edge components of $v$ are determined by the edges of the graph so that $v|_K=\chi_E$. The HA components of $v$ are determined by colored small circles in Figure \ref{fig:exa4}(middle). We represent the dates by colors. Let the first day correspond to red, the second day to blue and the third day to green. To represent which team is a home team at that day, we use small circles painted in the color of the day. A small circle near a vertex indicates that the team corresponding to the vertex is a home team in the day corresponding to the color of the small circle. Note that we do not draw  the vertices explicitly in the figure for simplicity.

Next, we explain how to represent an evidence indicating that it belongs to $\overline{\mathbf N}(PM(V))$ or ${\mathbf N}(PM(V))$. We use colored edges for that purpose. In Figure \ref{fig:exa4}(middle), the graph is double covered with edges. Note that, for each edge, if one end of the edge is at a home game, the other end is at an away game. There exist red edges $\{1,3\},\{1,4\},\{2,3\},\{2,4\}$. 
These edges represent two perfect matchings $\{1,4\},\{2,3\}$ and $\{1,3\},\{2,4\}$. Similarly, for each date, there exist two perfect matchings. Each perfect matching corresponds to $k_ip_i$ in the definition of ${\mathbf N}(PM(V))$. Then we have $2v=\sum_{p\in PM(V)} k_ip_i$ and $v|_K=\chi_E$. So, we see that $v\in \overline{\mathbf N}(PM(V))$. We say that such a graph has a double-edge covering consisting of perfect matchings compatible with the HAP table. 

Moreover, $v$ belongs to ${\mathbf N}(PM(V))$ since this graph has a 3-edge-coloring compatible with the HAP table as in Figure \ref{fig:exa4}(right).

\subsection{B-factorizability for regular graphs on 6 vertices}

Theoretically, the additional generators can be calculated by normaliz. But for a graph whose vertex set is rather large, the calculation needs enormous time. To begin with, we calculated additional generators for $|V|=6$. That is, in this subsection, we consider a scheduling problem with 6 teams.

\begin{ex}\label{ex:sixcrown}

The number of all the equal partitions $C$ on the set of size 6 is 10. The number of edges of the complete graph whose vertex set has size 6 is 15. So $R^{K\cup C}$ is a 25-dimensional linear space. $K_{3,3}$ has 6 perfect matchings. As it turned out, $PM(V)$ consists of 60 vectors. The Hilbert basis of $\overline{\mathbf N}(PM(V))$ can be calculated by normaliz. By using normaliz, we know that the number of additional generators are 90, which are all isomorphic.

An additional generator $v\in \mbox{problem}(V)$ is shown in Figure \ref{fig:add6}. This 4-regular graph on 6 vertices is isomorphic to the graph of the octahedron. Recall that the HAP table for four days is represented by the four colors of the small circles near each vertex. A colored small cycle indicates the home team on the date corresponding to the color. 

The graph in the figure has a double-edge covering of 8 perfect matchings. There exist two perfect matchings, whose edges are painted in the same color, per day in this figure. For two perfect matchings on the same date, we draw them at once in the same color in the figure. The double-edge covering means that $2v \in {\mathbf N}(PM(V))$. We have $v\in \overline{\mathbf N}(PM(V))$.

There exist no 4-edge-coloring compatible with the HAP table because we cannot take one perfect matching in the two perfect matchings for each day so that these four perfect matchings are disjoint.
So we have $v\notin {\mathbf N}(PM(V))$. This means that the graph of the octahedron is not B-factorizable by Lemma \ref{lem:addb}.

\begin{figure}[ht]
\begin{center}
\includegraphics[width=4cm]{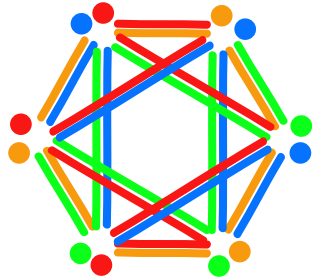}
\caption{Additional generator on 6 vertices}\label{fig:add6}
\end{center}
\end{figure}

\end{ex}

Note that calculating the Hilbert basis of the monoid $\overline{\mathbf N}(PM(V))$, which does not depend upon a graph,  naturally tells us that the graph of the octahedron is not B-factorizable.
Moreover, we have only to calculate the Hilbert basis once, that is, we do not have to calculate it for every HAP table although the direct test of Conjecture \ref{con:briskorn} needs calculations for all possible HAP tables.

\begin{thm}\label{thm:nocount6}
When $|V|=6$, there exists no counterexample to Conjecture \ref{con:dirk}.
\end{thm}

\begin{proof}
As for six teams, it suffices to consider only one additional generator because we know by calculation that all the additional generators are isomorphic. By Lemma \ref{lem:countha}, if there exists a counterexample to Conjecture \ref{con:dirk}, the HAP table of the counterexample contains four  HA components of the additional generator. That is, the entries of the HAP table for four days in the five days are determined by the additional generator. To complete a schedule for the five days, we have only to make entries of the HAP table for one more day. For each of the equal partitions, which are of 10 types, we have to check whether the vector obtained by adding it as one more day to the HAP table of the additional generator belongs to ${\mathbf N}(PM(V))$ or not. It is confirmed by calculation that they all belong to ${\mathbf N}(PM(V))$.
Figure \ref{fig:comp6} illustrates one case of the 10 cases as an example. Note that an equal partition is added to the HAP table on the fifth day as shown in the figure. In this case, there exists a 5-edge-coloring compatible with the HAP table.
\begin{figure}[ht]
\begin{center}
\includegraphics[width=4cm]{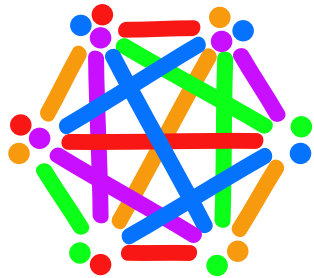}
\caption{5-edge-coloring compatible with HAP table}\label{fig:comp6}
\end{center}
\end{figure}
\end{proof}

\begin{prop}
$P(K_{|V|},HA)$ is a non-integral polytope for the HAP table in Figure \ref{fig:comp6}.
\end{prop}

\begin{proof}
We construct a non-integral point in ${\mathbf R}^{K\cup \{1,\ldots,5\}}$ in terms of the proof of Theorem \ref{thm:nocount6}.
Consider the point such that the value on each edge of the octahedron is 0 or 0.5 according to Figure \ref{fig:add6} and the value on each of the other three edges is 1 on the fifth day and 0 on the other day.
Then that point becomes a non-integral vertex of $P(K_{|V|},HA)$ because this point cannot move in the polytope when the components on the fifth day are fixed.
\end{proof}

So $P(K_{|V|},HA)$ is not an integral polytope although it contains an integral point.

\subsection{Monoid corresponding to a graph}

In this subsection, we introduce the monoid corresponding to a regular graph, and see a few examples of B-factorizable graphs.

In the case of $|V|=8$, the number of generating vectors of $\overline{\mathbf N}(PM(V))$ is so large that it is difficult to calculate its Hilbert basis by normaliz.  

To reduce the size of the problem, we consider the monoid corresponding to a regular graph. We consider the vectors $v\in PM(V)$ such that $E(v)$ is a matching of the given graph. That is, for a regular graph $(V,E)$, we consider the monoid $\overline{\mathbf N}(\{v\in PM(V)|E(v) \subset E\})$. We can rather easily calculate the Hilbert basis of the monoid corresponding to a regular graph with less edges than that corresponding to the complete graph because it has a smaller generating set. By calculating the Hilbert basis of the monoid for a regular graph, we may find an additional generator $v$ such that $E(v)$ is a subgraph of the given graph. If an additional generator $v\in \mbox{problem}(V)$ is found, graph $(V,E(v))$ turns out not to be B-factorizable by Lemma \ref{lem:addb}. If there exists no additional generator, the graph turns out to be B-factorizable by Lemma \ref{lem:countha}.

We take a glance at a few generators of the monoid corresponding to the graph of 3-dimensional cube. The upper part of Figure \ref{fig:cubeha} illustrates the equal partitions on the graph of 3-dimensional cube. There exist 6 types of equal partitions up to isomorphism. In the figure, small circles indicate home teams. For example, the second equal partition in the figure has three perfect matchings illustrated in the lower part of the figure. Each of them has the corresponding vector in $PM(V)$.

\begin{figure}[ht]
\begin{center}
\includegraphics[width=15cm]{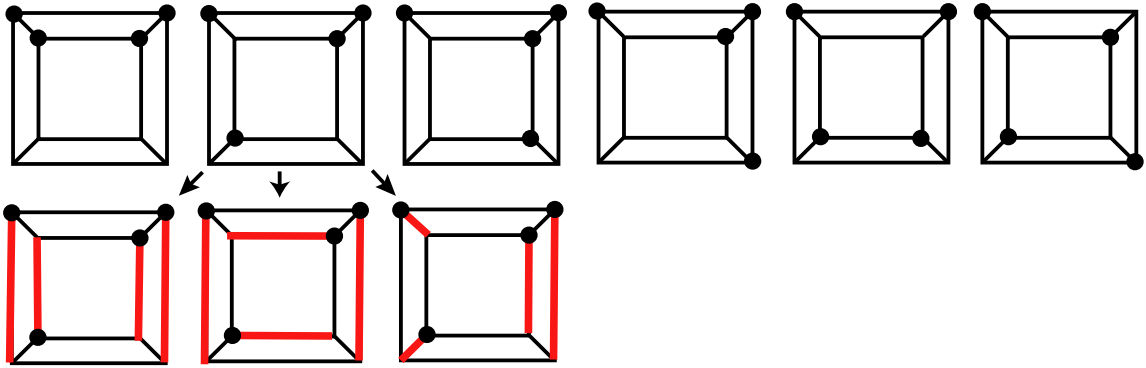}
\caption{Equal partitions on graph of 3-cube and generators of monoid}\label{fig:cubeha}
\end{center}
\end{figure}

We made sure as follows that the graph of the 3-dimensional cube is B-factorizable. Consider the vectors $v \in PM(V)$ such that $E(v)$ is a matching of the graph of 3-dimensional cube. We calculated the Hilbert basis of the monoid generated by these vectors.
We found that there exists no additional generator in the monoid. So the graph of the 3-dimensional cube turns out to be B-factorizable by Lemma \ref{lem:countha}.

\begin{figure}[ht]
\begin{center}
\includegraphics[width=3cm]{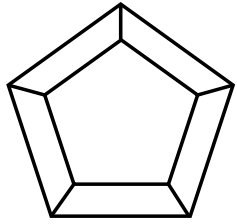}
\caption{Graph of pentagonal prism}\label{fig:candp}
\end{center}
\end{figure}

Similarly, we made sure that the graph of the pentagonal prism, illustrated in Figure \ref{fig:candp}, is B-factorizable by calculating the Hilbert basis of the monoid corresponding to the graph.

We may conjecture the following.

\begin{con}
The graph of every $k$-gonal prism is B-factorizable.
\end{con}

This conjecture seems to be easier than Conjecture \ref{con:dirk} because we have only to consider a HAP table for three days. But we have not solved it yet.

\subsection{B-factorizability for antiprism graphs}

For $|V|=8$, we found an example of graph which is not B-factorizable. In the next example, we show that the antiprism graph on 8 vertices is not B-factorizable. The antiprism is a circular graph in which $i$th and $(i+1)$th vertices are adjacent, and $i$th and $(i+2)$th vertices are adjacent for all $i$.

\begin{ex}\label{ex:anti8}
Consider $v\in \mbox{problem}(V)$ represented in Figure \ref{fig:anti8}. Then $E(v)$ is the antiprism graph on 8 vertices. Recall that the HA components of $v$ are described by the small circles near each vertex in the figure. The existence of a double-edge covering of this graph in the figure  indicates $2v\in {\mathbf N}(PM(V))$ and $v\in \overline{\mathbf N}(PM(V))$. Note that two perfect matchings in the same day are painted in the same color. 
Vector $v$ does not belong to ${\mathbf N}(PM(V))$ because of the following reason. We consider an intersection graph defined as follows for these eight perfect matchings. The vertices of the intersection graph correspond to $\{p\in PM(V)| p \leq v\}$. Two vertices $p_1$ and $p_2$ of the intersection graph are adjacent if the support of $p_1$ intersects with the support of $p_2$. In other words, two vertices of the intersection graph are adjacent if they have a common edge in the edge components or their dates are the same. Then there exists no stable set of size 4 in the intersection graph. This fact means that there exist no 4 perfect matchings disjoint with each other in our sense. That is, there exists no 4-edge-coloring using these matchings. So we have $v\notin {\mathbf N}(PM(V))$. So $v$ is an additional generator such that $E(v)$ is the antiprism on 8 vertices. Therefore the antiprism on 8 vertices is not B-factorizable by Lemma \ref{lem:addb}.

\begin{figure}[ht]
\begin{center}
\includegraphics[width=5cm]{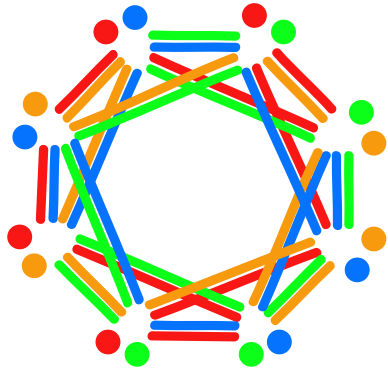}
\caption{Antiprism on 8 vertices}\label{fig:anti8}
\end{center}
\end{figure}
\begin{figure}[ht]
\begin{center}
\includegraphics[width=4cm]{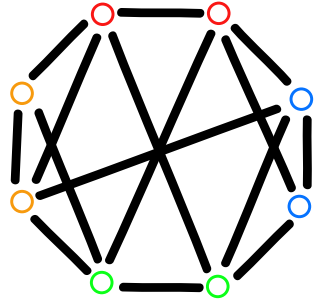}
\caption{Intersection graph of matchings}\label{fig:inter}
\end{center}
\end{figure}

\end{ex}

Recall that the non-B-factorizable graph on 6 vertices shown in Example \ref{ex:sixcrown} is also an antiprism. Generally, the following statement holds.

\begin{thm}\label{thm:antiprism}
Any antiprism on even vertices is not B-factorizable.
\end{thm}

\begin{proof}
The period of the schedule that we should consider is four days because the degree of the regular graph of the antiprism is 4. Construct a HAP table for 4 days. Then take $v\in \mbox{problem}(V)$ so that $E(v)$ is the antiprism and the HA components of $v$ correspond to the HAP table that we consider below. We show $v\in \overline{\mathbf N}(PM(V))$ and $v\notin {\mathbf
N}(PM(V))$.

Arrange $k=|V|$ vertices on a circle and number them as $\{0,...,k-1\}$ modulo $k$. First, consider the HAP table of the fixed one day. We begin with the HAP table of the one day whose entries consist of the simple alternation of H and A as HAHAHA...as in Figure \ref{fig:bsc8}. A pair of two vertices is said to be matchable if the pair corresponds to an edge of the graph, and one vertex is assigned to H and the other vertex is assigned to A. 

\begin{figure}[ht]
\begin{center}
\includegraphics[width=4cm]{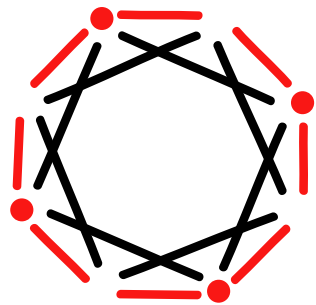}
\caption{Basic H and A arrangement}\label{fig:bsc8}
\end{center}
\end{figure}

For such a HAP table, a matchable pair of teams in the graph must be adjacent on the outer circle because two vertices at distance 2 are assigned to both H or both A. We paint the matchable edges in the figure. Then, we consider a slightly different HAP table from the original one. For example, by swapping H and A between vertices 0 and 1, the matchable pairs of  teams are $\{k-1,1\},\{0,1\},\{0,2\},\{1,3\},\{2,3\},...$.  Then the subgraph that consists of the colored edges has exactly one Hamilton cycle. The Hamilton cycle is twisted at a pair of the vertices whose labels, H and A, are swapped. The number of the perfect matchings which consist of matchable edges is exactly two.

In Figure \ref{fig:thc8}, a small red circle near a vertex means that the team is assigned to a home game on the fixed date. This arrangement is obtained by swapping H and A between the two vertices $\{0,1\}$ from the arrangement of alternating H and A simply.
Then the matchable edges, which connect between H and A vertices, are colored.

\begin{figure}[ht]
\begin{center}
\includegraphics[width=5cm]{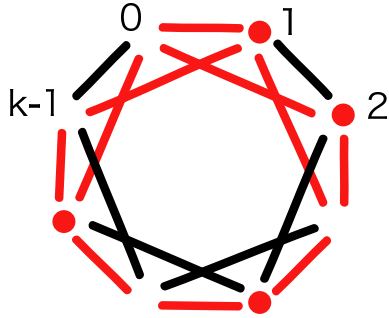}
\caption{Swapping H and A for one pair}\label{fig:thc8}
\end{center}
\end{figure}

We want to make a HAP table by swapping H and A for a pair of adjacent vertices on the outer circle so that it has a Hamilton cycle consisting of matchable edges.
For a fixed day, we cannot swap a pair at a distance of at most 2 from the swapping pair because the subgraph consisting of the matchable edges does not have any Hamilton cycle. For example, we cannot swap H and A between 0 and 1, and simultaneously swap H and A between 2 and 3. The next figure shows an example obtained by swapping H and A for each of pairs at a distance of three. It has a Hamilton cycle of length 8. The Hamilton cycle consists of the edges in two perfect matchings. We regard the distance between pair $\{0,1\}$ and pair $\{2,3\}$ as 2.  As you see, by swapping H and A for some pairs so that each pair of swapping pairs has a distance of at least 3, the number of perfect matchings consisting of matchable edges keeps exactly 2.

\begin{figure}[ht]
\begin{center}
\includegraphics[width=4cm]{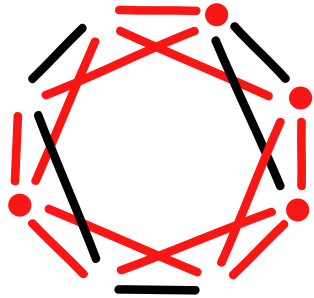}
\caption{Swapping H and A for two pairs}\label{fig:tht8}
\end{center}
\end{figure}

There exist $k$ edges $\{i,i+1\}$ on the outer circle. We want to swap H and A for any adjacent pair on the outer circle at once for the four days so that the antiprism has a double-edge covering consisting of perfect matchings compatible with the HAP table. If such a HAP table exists, we have $2v\in {\mathbf N}(PM(V))$ and $v\in \overline{\mathbf N}(PM(V))$.

We show that we can obtain a desired swapping. 
We have already discussed the antiprism on 6 vertices, which is not B-factorizable, in Example \ref{ex:sixcrown}. So we consider antiprisms on at least 8 vertices in the sequel.

Consider the antiprism on 8 vertices. To twist 8 pairs totally, we have only to twist two pairs per day. You may simply think that it suffices to, for example, twist $\{0,1\}$ and $\{4,5\}$ on the first day, $\{1,2\}$ and $\{5,6\}$ on the second day, $\{2,3\}$ and $\{6,7\}$ on the third day, and $\{3,4\}$ and $\{7,8\}$ on the last day. But this twisting is not suitable because we can take four disjoint matchings compatible with the HAP table and we have $v\in {\mathbf N}(PM(V))$. So we have to consider a bit different HAP table. Twist $\{1,2\}$ and $\{4,5\}$ on the first day, $\{0,1\}$ and $\{5,6\}$ on the second day, $\{2,3\}$ and $\{6,7\}$ on the third day, and $\{3,4\}$ and $\{7,8\}$ on the last day from the original one. Then we have $v\in \overline{\mathbf N}(PM(V))$. Since there exist no 4-edge-coloring compatible with this HAP table, we have 
$v\notin {\mathbf N}(PM(V))$.

Next we consider the case of 10 vertices. We twist pairs $\{0,1\}$, $\{3,4\}$ and $\{6,7\}$ on the first day, and pairs $\{1,2\}$, $\{5,6\}$ and $\{8,9\}$, and pairs $\{2,3\}$ and $\{7,8\}$ on the third day, and pairs $\{4,5\}$ and $\{9,0\}$ on the fourth day. Note that each pair of the adjacent vertices on the outer circle is swapped exactly once. Then we have $2v\in {\mathbf N}(PM(V))$ and $v\in \overline{\mathbf N}(PM(V))$. Since there exists no set of four perfect matchings disjoint with each other in the meaning of Example \ref{ex:anti8}, we have $v\notin {\mathbf N}(PM(V))$.

\begin{figure}[ht]
\begin{center}
\includegraphics[width=5cm]{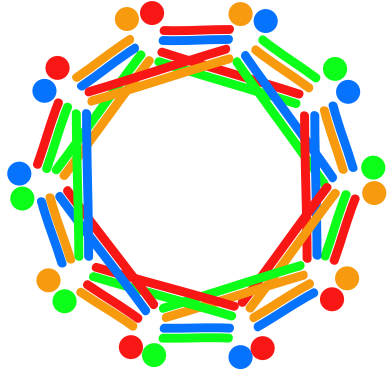}
\caption{Antiprism on 10 vertices}\label{fig:anti10}
\end{center}
\end{figure}

Next we distinguish two cases according to whether or not the number of vertices is a multiple of 4.

Consider the case that the number of vertices is a multiple of 4. Similarly  to the case of 8 vertices, we add 4 vertices repeatedly to the HAP table of the graph on 8 vertices. In the case of $4k$ vertices, we swap H and A for $k$ pairs for each day.

Consider the case that the number of vertices is not a multiple of 4. We add 4 vertices repeatedly to the example on 10 vertices. In the case of $4k+2$ vertices, swap H and A for $k+1$ pairs on the first day and the second day, and $k$ pairs on the third day and the fourth day.

\end{proof}

\subsection{B-factorizability for complete bipartite graphs}

In this subsection, we investigate whether a complete bipartite graph whose partite sets have the same size is B-factorizable or not. 

We already know whether a given $r$-regular graph on 6 vertices is B-factorizable or not by calculating the Hilbert basis of ${\mathbf N}(PM(V))$ on 6 vertices. So the complete bipartite graph $K_{3,3}$ is B-factorizable.

\begin{thm}\label{thm:k44}
The complete bipartite graph $K_{4,4}$ is B-factorizable.
\end{thm}

\begin{proof}
The proof relies on the calculation by normaliz.

We may calculate the Hilbert basis of the monoid generated by the vectors in $PM(V)$ such that the edge component corresponds to a perfect matching of $K_{4,4}$. If there exists no additional generator, it turns out that the graph is B-factorizable.

But, the generating vectors of the monoid for the complete bipartite graph $K_{4,4}$ are too large to calculate at one time by normaliz. We divide it into cases to reduce the generating vectors. 

The equal partitions may be classified into three types according to the number of home games in one partite set of the complete bipartite graph as in Figure \ref{fig:bib44}. In the figure, a matchable edge, which connects between H and A, is colored.

\begin{figure}[ht]
\begin{center}
\includegraphics[width=9cm]{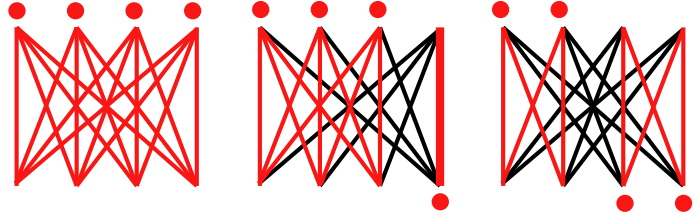}
\caption{Three types of equal partitions of $K_{4,4}$}\label{fig:bib44}
\end{center}
\end{figure}

When the HAP table contains an equal partition of the second type, one edge is always contained in all perfect matchings for the equal partition of the second type. That edge is illustrated as a thick edge in the figure. In this case, we may omit the perfect matchings containing that edge on the other day. So, we calculate, by normaliz, the Hilbert basis of the monoid generated by the vectors in $PM(V)$ such that its edge components do not contain that edge. By calculation, we found that there exists no additional generator in the monoid.

When the HAP table contains no equal partition of the second type, we can calculate, by normaliz, the Hilbert basis of the monoid generated by vectors such that its HA component is an equal partition of the first type or the third type.
By calculation, we found that there exists no additional generator in the monoid.
 
So we confirmed that $v \in \overline{\mathbf N}(PM(V))\cap \mbox{problem}(V)$ such that $E(v)$ is $K_{4,4}$ implies $v \in{\mathbf N}(PM(V))$. So $K_{4,4}$ turned out to be B-factorizable.

\end{proof}

We still do not know whether the complete bipartite graphs with more vertices than $K_{4,4}$ are B-factorizable or not. However, we may conjecture the following.

\begin{con}
Every complete bipartite graph $K_{k,k}$ is B-factorizable.
\end{con}

\subsection{B-factorizability for cubic graphs}

What class of regular graphs consists of B-factorizable graphs? Is any cubic graph B-factorizable? The answer is no. There exists a cubic graph which is not B-factorizable. Moreover we can make a cubic bipartite graph which is not B-factorizable. In this subsection, we investigate B-factorizability for cubic graphs.

\begin{ex}\label{ex:petersen}
This example shows that the Petersen graph is not B-factorizable. 
Consider $v\in \mbox{problem}(V)$ represented in Figure \ref{fig:pet3}. Then $E(v)$ is the Petersen graph. Since there exists a double-edge covering of this graph compatible with the HAP table, we have $2v \in{\mathbf N}(PM(V))$ and $v \in\overline{\mathbf N}(PM(V))$. On the other hand, $v\notin{\mathbf N}(PM(V))$ because the Petersen graph does not have any 3-edge-coloring. Therefore $v$ becomes an additional generator, and the Petersen graph is not B-factorizable.

\begin{figure}[ht]
\begin{center}
\includegraphics[width=5cm]{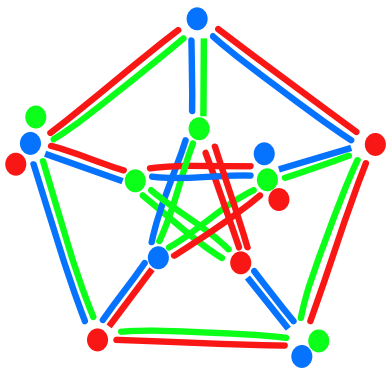}
\caption{Petersen graph with a double-edge covering}\label{fig:pet3}
\end{center}
\end{figure}

\end{ex}

The following example shows that the converse of Lemma \ref{lem:addb} does not hold. This example is not a cubic graph but related with the Petersen graph, which is cubic.

\begin{ex}\label{ex:petersen1}
Consider $v\in \mbox{problem}(V)$ such that $E(v)$ is the graph expressed in Figure \ref{fig:pet2} and the HA components are determined by the colors of the small circles in the figure. The perfect matchings in the figure form a double-edge covering of the graph. So it turns out that $2v\in {\mathbf N}(PM(V))$ and $v\in \overline{\mathbf N}(PM(V))$. 

The edges colored orange cannot be colored other than orange.  Note that we obtain the Petersen graph by removing the orange edges from the graph.
So there exists no 4-edge-coloring compatible with the HAP table. Therefore we have $v\notin{\mathbf N}(PM(V))$. So the graph corresponding to the figure is not B-factorizable. 

We divide $v$ into $v_1=(\mbox{the perfect matching colored orange, orange day in the HAP table})$ and $v_2$, corresponding to the Petersen graph, so that $v=v_1+v_2$. Then $v_1$ and $v_2$ belong to 
$\overline{\mathbf N}(PM(V))$. So $v$ is not an additional generator since $v$ does not belong to the Hilbert basis.

\begin{figure}[ht]
\begin{center}
\includegraphics[width=5cm]{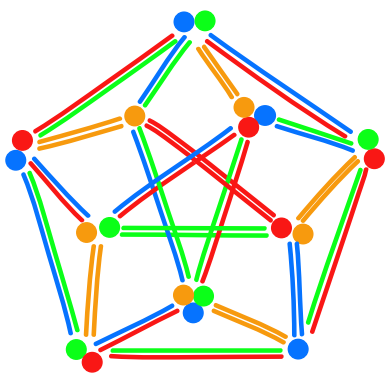}
\caption{Non-B-factorizable graph without additional generators}\label{fig:pet2}
\end{center}
\end{figure}
\end{ex}

The next example shows that there exists a 3-regular bipartite graph which is not B-factorizable.

\begin{ex}\label{ex:cubicbib}

Consider the cubic bipartite graph on 20 vertices illustrated in Figure \ref{fig:3reg}. Consider the vector $v\in\mbox{problem}(V)$ represented by the graph and the colored small circles in the figure.
Since there exists a double-edge covering of the graph compatible with the HAP table, we have $2v\in {\mathbf N}(PM(V))$ and $v\in \overline{\mathbf N}(PM(V))$. 

Next, we consider whether $v$ belongs to ${\mathbf N}(PM(V))$ or not. We try to make a 3-edge-coloring of the graph compatible with the HAP table. It is easy to see that any edge that has double edges in the same color cannot have another color in any 3-edge-coloring. Therefore the possible colors of each of the other edges are narrowed down to two colors. But we cannot complete edge-coloring by picking up one color from the possible two colors. Choosing one color in the two colors induces a contradiction. So $v$ does not belong to ${\mathbf N}(PM(V))$. So this graph is not B-factorizable. Since $v$ cannot be divided into two distinct vectors in $\overline{\mathbf N}(PM(V))$, $v$ is an additional generator.

\begin{figure}[ht]
\begin{center}
\includegraphics[width=8cm]{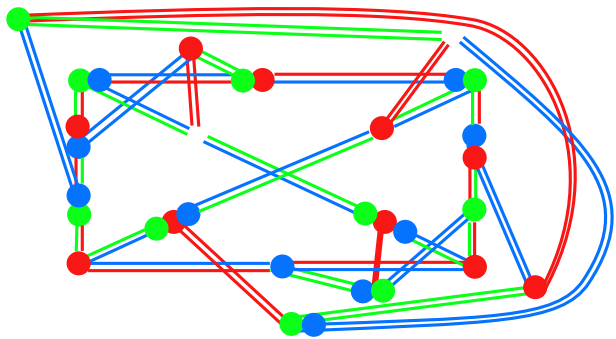}
\caption{Cubic bipartite graph which is not B-factorizable}\label{fig:3reg}
\end{center}
\end{figure}
\end{ex}

\section{Concluding remarks}

To summarize our discussions so far, the B-factorizability for a $k$-regular graph with a HAP table is related with $k$-edge-coloring compatible with the HAP table. Moreover, we point out here that it may be related with the arguments about the $k$-edge-colorability for $k$-regular graphs in the standard graph theory, i.e. without the HAP table, though it is not known that one statement implies directly another statement. 

\begin{con}(\cite{goddyn})
For any graph, the characteristic vectors $\chi_q$ such that $q$ is a perfect matching form the Hilbert basis of the monoid generated by them when the graph does not have the Petersen graph as a graph minor.
\end{con}

The statement of this conjecture is stronger than the following result.

\begin{thm}(\cite{RSST})
Every bridgeless cubic graph without a Petersen minor is 3-edge-colorable.
\end{thm}

The cubic bipartite graph shown in Example \ref{ex:cubicbib} has a Petersen minor. We do not know whether there exists a cubic graph without a Petersen minor that is not B-factorizable.

Every example, appeared in this paper, of a regular graph which is not B-factorizable has a double-edge covering compatible with the HAP table. 
We want to know whether any regular graph which has a fractional covering compatible with the HAP table must have a double-edge covering compatible with the HAP table. That is, we propose the following conjecture.

\begin{con}
For any $v \in \overline{\mathbf N}(PM(V))\cap \mbox{problem}(PM(V))$, $2v\in {\mathbf N}(PM(V))$ holds.
\end{con}

The Berge-Fulkerson conjecture below seems to be related with the statement above.

\begin{con}(\cite{BF})
If $G$ is a bridgeless cubic graph, then there exist 6 perfect matchings $M_1, M_2, M_3, M_4, M_5, M_6$ of with the property that every edge of $G$ is contained in exactly two of $M_1, M_2, M_3, M_4, M_5, M_6$.
\end{con}


\begin{thebibliography}{1}
\bibitem{dirk0}
D. Briskorn, {Feasibility of Home-Away-Pattern Sets: A Necessary Condition}, available at http://www.bwl.uni-kiel.de/Prod/team/doktoranden/briskorn/HapSetFeasibility2007.pdf, 2007.
\bibitem{dirk}
D. Briskorn, {Feasibility of home-away-pattern sets for round robin tournaments}, Operations Research Letters 36 (2008) 283-284.
\bibitem{normaliz} W. Bruns and R. Koch, Normaliz--a program for computing normalizations of affine semigroups, 1998. Available via anonymous ftp from ftp.mathematik.Uni-Osnabrueck.DE/pub/osm/kommalg/software.
\bibitem{BF} D.R. Fulkerson, Blocking and anti-blocking pairs of polyhedra, Math. Programming 1 (1971) 168-194.
\bibitem{goddyn}
Luis A. Goddyn, {
Cones, Lattices and Hilbert Bases of Circuits and. Perfect Matchings}, in ``Graph Structure Theory'', Contemporary Mathematics 147 (1993), 419-439.
\bibitem{LP}
L. Lovasz and M. D. Plummer. Matching Theory. North-Holland, Amsterdam, 1986.
\bibitem{miya}
R. Miyashiro, H. Iwasaki, and T. Matsui,
 {Characterizing Feasible Pattern Sets
  with a Minimum Number of Breaks
in: E. Burke and P. De Causmaecker (eds.),
Practice and Theory of Automated Timetabling IV,
Selected Revised Papers},
Lecture Notes in Computer Science 2740,
Springer, 2003, pp. 78--99.
\bibitem{RSST}
N. Robertson, D. Sanders, P. Seymour, and R. Thomas, The four-color theorem, J. Combin. 
Theory Ser. B 70 (1997), 2--44.
\bibitem{Werra80}
D. de Werra, {Geograph, Games and Graphs}, Discrete Applied Mathematics 2
	(1980) pp. 327--337.
\bibitem{Werra82}
D. de Werra, {Minimizing Irregularities in Sports Schedules Using Graph
	Theory,} Discrete Applied Mathematics 4 (1982) pp. 217--226.
\bibitem{Werra85}
D. de Werra, {On the Multiplication of Divisions: the Use of Graphs for
	Sports Scheduling,} Networks 15 (1985) pp. 125--136.
\bibitem{Werra88}
D. de Werra, {Some Models of Graphs for Scheduling Sports Competitions,}
Discrete Applied Mathematics 21 (1988) pp. 47--65.
\end{thebibliography}
\end{document}